\documentclass[a4paper,12pt]{amsproc}
\usepackage{amsmath,amssymb,cite,mathptmx}

\makeatletter
\newcommand{\subjclassname@later}{\textup{2010} Mathematics Subject Classification}
\makeatother

\newtheorem{theorem}{Theorem}[section]
\newtheorem{lemma}{Lemma}[section]

\newtheorem{remark}{Remark}[section]

\numberwithin{equation}{section} \numberwithin{theorem}{section}

\DeclareMathOperator{\RE}{Re}

\begin{document}
\title{Close-to-convexity and Starlikeness of Analytic Functions}

\author[S. K. Lee]{See Keong Lee}
\address{School of Mathematical Sciences, Universiti Sains Malaysia,
11800 USM,  Penang,  Malaysia} \email{sklee@cs.usm.my }

\author[V.  Ravichandran]{ V.  Ravichandran}
\address{Department of Mathematics, University of Delhi, Delhi 110
007, India\\School of Mathematical Sciences, Universiti Sains Malaysia, 11800 USM, Penang, Malaysia} \email{vravi@maths.du.ac.in}

\author[S. Shamani]{Shamani Supramaniam}
\address{School of Mathematical Sciences,
Universiti Sains Malaysia,  11800 USM,  Penang,  Malaysia}
\email{sham105@hotmail.com}


\begin{abstract}For functions  $f(z)=z^p+a_{n+1}z^{p+1}+\cdots$ defined on the open unit disk, the condition $\Re (f'(z)/z^{p-1})>0$ is sufficient for close-to-convexity of $f$. By making use of this result, several sufficient conditions for close-to-convexity are investigated and relevant connections with previously known results are indicated.
\end{abstract}


\maketitle

\section{Introduction}
Let $\mathbb{D} : = \{ z \in \mathbb{C} : |z| < 1 \}$ be the open
unit disk and  $\mathcal{A}_{p,n}$ be the class of all analytic
functions $f:\mathbb{D}\rightarrow \mathbb{C}$ of the form $f(z)=z^{p}
+a_{n+p}z^{n+p}+a_{n+p+1}z^{n+p+1}+\dotsc$ with
$\mathcal{A}:=\mathcal{A}_{1,1}$. For studies related to multivalent functions,   see \cite{living,nuno1,nuno2,nuno3,nuno4}.
Singh and Singh \cite{singh} obtained  several interesting conditions for functions $f\in \mathcal{A}$ satisfying  inequalities involving $f'(z)$ and $zf''(z)$ to be univalent  or starlike in $\mathbb{D}$. Owa \emph{et al.} \cite{owa} generalized the results
of Singh and Singh \cite{singh} and also obtained several sufficient conditions for close-to-convexity, starlikeness and convexity of
functions $f\in\mathcal{A}$. In fact, they have proved the following theorems.

\begin{theorem}{\rm \cite[Theorems 1-3]{owa}}  \label{th1.1} Let $ 0\leq\alpha<1$ and $\beta,\gamma \geq 0$.
If  $f\in \mathcal{A}$, then
\[\RE
\left(1+\frac{zf''(z)}{f'(z)}\right)>\frac{1+3\alpha}{2(1+\alpha)} \Longrightarrow \RE
\left(f'(z)\right)>\frac{1+\alpha}{2},\]
 \[\RE
\left(1+\frac{zf''(z)}{f'(z)}\right)< \frac{3+2\alpha}{(2+\alpha)}\Longrightarrow
\left|f'(z)-1\right|<1+\alpha,\]
\[\left|f'(z)-1\right|^\beta |zf''(z)|^\gamma
<\frac{(1-\alpha)^{\beta+\gamma}}{2^{\beta+2\gamma}}
 \Longrightarrow \RE
\left(f'(z)\right)>\frac{1+\alpha}{2}.\]
\end{theorem}

\begin{theorem} {\rm \cite[Theorem 4]{owa}}  \label{th1.2} Let $1<\lambda<3$.
If  $f\in\mathcal{A}$, then
\[\RE
\left(1+\frac{zf''(z)}{f'(z)}\right)<\left\{
                                       \begin{array}{ll}
                                        \frac{5\lambda-1}{2(\lambda+1)}, & \hbox{$1<\lambda \leq2$;}
\\[10pt]
                                         \frac{\lambda+1}{2(\lambda-1)}, & \hbox{$2<\lambda <3$,}
                                       \end{array}
                                     \right.
 \Longrightarrow\frac{zf'(z)}{f(z)}\prec
 \frac{\lambda(1-z)}{\lambda-z}.\]
\end{theorem}

In this present paper, the above results are extended for functions $f\in\mathcal{A}_{p,n}$ and in particular for
functions in $\mathcal{A}_{1,n}$.

\section{Close-to-convexity and Starlikeness }
For $f\in\mathcal{A}$, the condition $\RE f'(z)>0$ implies close-to-convexity and univalence of $f$. Similarly, for $f\in \mathcal{A}_{p,1}$, the inequality $\RE (f'(z)/z^{p-1})>0$ implies $p$-valency of $f$. See \cite{umetoh, umepam}. From this result, the functions satisfying the hypothesis of Theorems \ref{th2.1}--\ref{th2.3} are $p$-valent in $\mathbb{D}$. A function $f\in \mathcal{A}_{p,1}$ is close-to-convex if there is a $p$-valent convex function $\phi$ such that $\RE (f'(z)/\phi(z))>0$. Also they are all close-to-convex with respect to $\phi(z)=z^p$.

\begin{theorem} \label{th2.1} If the function $f\in\mathcal{A}_{p,n}$ satisfies the
inequality \begin{equation}\label{hpy1}  \RE
\left(1+\frac{zf''(z)}{f'(z)}\right)>\frac{(2p-n)+\alpha(2p+n)}{2 (\alpha+1)},\end{equation}
then \[\RE
\left(\frac{f'(z)}{pz^{p-1}}\right)>\frac{1+\alpha}{2}.\]\end{theorem}

For the proof of our main results, we need the following lemma.
\begin{lemma}\label{lem1.1} {\rm \cite[Lemma 2.2a]{miller}}
Let $z_0\in \mathbb{D}$ and $r_0=|z_0|$. Let $f(z)=a_n z^n
+a_{n+1}z^{n+1}+\cdots$ be continuous on
$\mathbb{\overline{D}}_{r_0}$ and analytic on
$\mathbb{D}_{r_0}\cup\{z_0\}$ with $f(z)\not\equiv0$ and $n\geq 1$.
If \[|f(z_0)|=\max\{|f(z)|:z\in \mathbb{\overline{D}}_{r_0}\},\]
then there exists an $m\geq n$ such that \begin{enumerate}
                                           \item $\displaystyle\frac{z_0
                                           f'(z_0)}{f(z_0)}=m$, and
                                           \item $\displaystyle\RE \frac{z_0
                                           f''(z_0)}{f'(z_0)}+1\geq
                                           m$.
                                         \end{enumerate}
\end{lemma}

\begin{proof}[Proof of Theorem~\ref{th2.1}]
Let the function $w$ be defined by
\begin{equation} \label{eq1}\frac{f'(z)}{pz^{p-1}}=\frac{1+\alpha w(z)}{1+w(z)}.\end{equation} Then
clearly $w$ is analytic in $\mathbb{D}$ with $w(0)= 0$. From
\eqref{eq1}, some computation yields
\begin{equation} \label{eq2}1+\frac{zf''(z)}{f'(z)}=p+\frac{\alpha zw'(z)}{1+\alpha w(z)}-\frac{ zw'(z)}{1+ w(z)}.
\end{equation}
Suppose there exists a point $z_0 \in\mathbb{D}$ such that
\[|w(z_0)|=1\ \text{and}\ |w(z)|<1\ \text{when}\ |z|<|z_0|.\]
Then by applying Lemma \ref{lem1.1}, there exists $m\geq n$ such that
\begin{equation} \label{eq3} z_0 w'(z_0)=mw(z_0), \quad (w(z_0)=e^{i\theta}; \theta\in \mathbb{R}).\end{equation} Thus, by
using \eqref{eq2} and \eqref{eq3}, it follows that
\begin{align*} \RE\left(1+\frac{z_0 f''(z_0)}{f'(z_0)}
\right)&=p+ \RE \left(\frac{\alpha m w(z_0)}{1+\alpha
w(z_0)}\right)- \RE\left(\frac{mw(z_0)}{1+w(z_0)}\right)
\\&=p+ \RE\left(\frac{\alpha m e^{i\theta}}{1+\alpha
e^{i\theta}}\right)- \RE\left(\frac{ m
e^{i\theta}}{1+e^{i\theta}}\right)\\&=p+ \frac{\alpha
m(\alpha+\cos
\theta)}{1+\alpha^2+2\alpha\cos\theta}-\frac{m}{2}\\&\leq
\frac{(2p-n)+\alpha(2p+n)}{2(\alpha+1)}
\end{align*} which contradicts the hypothesis \eqref{hpy1}. It follows that
$|w(z)|<1,$ that is
\[\left|\frac{1-\frac{f'(z)}{pz^{p-1}}}{\frac{f'(z)}{pz^{p-1}}-\alpha}\right|<1.\]
This evidently completes the proof of Theorem \ref{th2.1}.
\end{proof}

Owa \cite{owabam} shows that a function $f\in\mathcal{A}_{p,1} $ satisfying $\RE (1+zf''(z)/f'(z))<p+1/2$ implies $f$ is $p$-valently starlike. Our next theorem investigates the close-to-convexity of this type of functions. For related results, see \cite{owaaml,liusjm, yama}.

\begin{theorem}\label{th2.2}
If the function $f\in\mathcal{A}_{p,n}$ satisfies the inequality
\begin{equation}\label{hpy2}  \RE
\left(1+\frac{zf''(z)}{f'(z)}\right)< \frac{(p+n)\alpha
+(2p+n)}{ (\alpha+2)},\end{equation} then \[
\left|\frac{f'(z)}{pz^{p-1}}-1\right|<1+\alpha.\]
\end{theorem}

\begin{proof}
Consider the function $w$ defined by \begin{equation} \label{eq2a}\frac{f'(z)}{pz^{p-1}}=(1+\alpha)w(z)+1.\end{equation}
Then
clearly $w$ is analytic in $\mathbb{D}$ with $w(0)= 0$. From
\eqref{eq2a}, some computation yields
\begin{equation} \label{eq3a}1+\frac{zf''(z)}{f'(z)}=p+\frac{(1+\alpha)zw'(z)}{(1+\alpha) w(z)+1}.
\end{equation}
Suppose there exists a point $z_0 \in\mathbb{D}$ such that
\[|w(z_0)|=1\ \text{and}\ |w(z)|<1\ \text{when}\ |z|<|z_0|.\]
Then by applying Lemma \ref{lem1.1}, there exists $m\geq n$ such that
\begin{equation} \label{eq4a} z_0 w'(z_0)=mw(z_0), \quad (w(z_0)=e^{i\theta}; \theta\in \mathbb{R}).\end{equation} Thus, by
using \eqref{eq3a} and \eqref{eq4a}, it follows that
\begin{align*}\RE\left(1+\frac{z_0 f''(z_0)}{f'(z_0)}
\right)&=p+\RE \left(\frac{(1+\alpha)z_{0}w'(z_{0})}{(1+\alpha) w(z_{0})+1}\right)\\
&=p+ \RE\left(\frac{(1+\alpha)me^{i\theta}}{(1+\alpha) e^{i\theta}+1}\right)\\
&=p+ \frac{m(1+\alpha)(1+\alpha+\cos
\theta)}{1+(1+\alpha)^2+2(1+\alpha)\cos\theta} \\
&\geq \frac{(p+n)\alpha+(2p+n)}{(\alpha+2)},
\end{align*} which contradicts the hypothesis \eqref{hpy2}. It follows that
$|w(z)|<1,$ that is,
\[\left|\frac{f'(z)}{pz^{p-1}}-1\right|<1+\alpha.\]
This evidently completes the proof of Theorem \ref{th2.2}.
\end{proof}

   Owa \cite{owasin} has also showed that a function $f\in\mathcal{A}$ satisfying $|f'(z)/g'(z)-1|^\beta|zf''(z)/g'(z)-zf'(z)g''(z)/(g'(z))^2|^\gamma<(1+\alpha)^{\beta+\alpha}$, for $0\leq\alpha<1$, $\beta\geq0$, $\gamma\geq0$ and $g$  a convex function, is close-to-convex. Also, see \cite{linnmj}. Our next theorem investigates the close-to-convexity of similar class of functions.

\begin{theorem} \label{th2.3}
If $f\in\mathcal{A}_{p,n}$, then
\begin{equation}\label{hyp3}\left|\frac{f'(z)}{pz^{p-1}}-1\right|^\beta \left|\frac{f''(z)}{z^{p-2}}-(p-1)\frac{f'(z)}{z^{p-1}}\right|^\gamma
<\frac{(pn)^\gamma(1-\alpha)^{\beta+\gamma}}{2^{\beta+2\gamma}}
 \end{equation} implies \[\RE
\left(\frac{f'(z)}{pz^{p-1}}\right)>\frac{1+\alpha}{2},\] and \begin{equation}\label{hypne}\left|\frac{f'(z)}{pz^{p-1}}-1\right|^\beta \left|\frac{f''(z)}{z^{p-2}}-(p-1)\frac{f'(z)}{z^{p-1}}\right|^\gamma
<(pn)^\gamma|1-\alpha|^{\beta+\gamma}
 \end{equation} implies \[
\left|\frac{f'(z)}{pz^{p-1}}-1\right|<1-\alpha.\]
\end{theorem}

\begin{proof}
For the function $w$ defined by
\begin{equation}\label{eqw3}
\frac{f'(z)}{pz^{p-1}}=\frac{1+\alpha w(z)}{1+w(z)},
\end{equation}
we can rewrite \eqref{eqw3} to yield \[\frac{f'(z)}{pz^{p-1}}-1=\frac{(\alpha-1)w(z)}{1+w(z)}\], which leads to
\begin{equation}\label{eqb}\left|\frac{f'(z)}{pz^{p-1}}-1\right|^\beta=\frac{|w(z)|^\beta|1-\alpha|^\beta}{|1+w(z)|^\beta}.\end{equation} By some computation, it is evident that \[\frac{f''(z)}{z^{p-2}}-(p-1)\frac{f'(z)}{z^{p-1}}=\frac{p(\alpha-1)zw'(z)}{(1+w(z))^2}\] or
\begin{equation}\label{eqg}\left|\frac{f''(z)}{z^{p-2}}-(p-1)\frac{f'(z)}{z^{p-1}}\right|^\gamma=\frac{p^\gamma|zw'(z)|^\gamma|1-\alpha|^\gamma}{|1+w(z)|^{2\gamma}}.\end{equation}
From \eqref{eqb} and \eqref{eqg}, it follows that
\[\left|\frac{f'(z)}{pz^{p-1}}-1\right|^\beta \left|\frac{f''(z)}{z^{p-2}}-(p-1)\frac{f'(z)}{z^{p-1}}\right|^\gamma=\frac{p^\gamma|w(z)|^\beta(1-\alpha)^{\beta+\gamma}|zw'(z)|^\gamma}{|1+w(z)|^{\beta+2\gamma}}.\]
Suppose there exists a point $z_0 \in\mathbb{D}$ such that
\[|w(z_0)|=1\ \text{and}\ |w(z)|<1\ \text{when}\ |z|<|z_0|.\]
Then \eqref{eq3} and Lemma \ref{lem1.1} yield
\begin{align*}
\left|\frac{f'(z_0)}{pz_{0}^{p-1}}-1\right|^\beta \left|\frac{f''(z_{0})}{z_{0}^{p-2}}-(p-1)\frac{f'(z_{0})}{z_{0}^{p-1}}\right|^\gamma
&=\frac{p^\gamma(1-\alpha)^{\beta+\gamma}|w(z_0)|^\beta
|mw(z_0)|^\gamma}{|1+e^{i\theta}|^{\beta+2\gamma}}\\&=\frac{p^\gamma
m^\gamma
(1-\alpha)^{\beta+\gamma}}{(2+2\cos\theta)^{(\beta+2\gamma)/2}}\\&\geq
\frac{p^\gamma n^\gamma
(1-\alpha)^{\beta+\gamma}}{2^{\beta+2\gamma}},
\end{align*} which contradicts the hypothesis \eqref{hyp3}. Hence
$|w(z)|<1$, which implies
\[\left|\frac{1-\frac{f'(z)}{pz^{p-1}}}{\frac{f'(z)}{pz^{p-1}}-\alpha}\right|<1,\] or equivalently \[\RE
\left(\frac{f'(z)}{pz^{p-1}}\right)>\frac{1+\alpha}{2}.\]
For the second implication in the proof, consider the function $w$ defined by
\begin{equation}\label{eqwnew}
\frac{f'(z)}{pz^{p-1}}=1+(1-\alpha) w(z).
\end{equation}
Then \begin{equation}\label{eqb2}\left|\frac{f'(z)}{pz^{p-1}}-1\right|^\beta=|1-\alpha|^{\beta}|w(z)|^\beta\end{equation} and
\begin{equation}\label{eqg2}\left|\frac{f''(z)}{z^{p-2}}-(p-1)\frac{f'(z)}{z^{p-1}}\right|^\gamma=p^\gamma|zw'(z)|^\gamma|1-\alpha|^\gamma.
\end{equation}
From \eqref{eqb2} and \eqref{eqg2}, it is clear that \[\left|\frac{f'(z)}{pz^{p-1}}-1\right|^\beta \left|\frac{f''(z)}{z^{p-2}}-(p-1)\frac{f'(z)}{z^{p-1}}\right|^\gamma=p^\gamma |w(z)|^\beta|1-\alpha|^{\beta+\gamma}|zw'(z)|^\gamma.\]
Suppose there exists a point $z_0 \in\mathbb{D}$ such that
\[|w(z_0)|=1\ \text{and}\ |w(z)|<1\ \text{when}\ |z|<|z_0|.\]
Then by applying Lemma \ref{lem1.1} and using \eqref{eq3}, it follows that
\begin{align*}
\left|\frac{f'(z_0)}{pz_{0}^{p-1}}-1\right|^\beta \left|\frac{f''(z_{0})}{z_{0}^{p-2}}-(p-1)\frac{f'(z_{0})}{z_{0}^{p-1}}\right|^\gamma
&=p^\gamma |w(z_{0})|^\beta|1-\alpha|^{\beta+\gamma}|z_{0}w'(z_{0})|^\gamma\\&=p^\gamma
m^\gamma
|1-\alpha|^{\beta+\gamma}\\&\geq
(p n)^\gamma
|1-\alpha|^{\beta+\gamma},
\end{align*} which contradicts the hypothesis \eqref{hypne}. Hence
$|w(z)|<1$ and this implies
\[\left|\frac{f'(z)}{pz^{p-1}}-1\right|<1-\alpha\].
Thus the proof is complete.
\end{proof}

In next theorem, we need the concept of subordination. Let $f$ and $g$ be analytic functions defined on $\mathbb{D}$. Then $f$ is \emph{subordinate} to $g$, written $f\prec g$, provided there is an analytic function $w:\mathbb{D}\rightarrow \mathbb{D}$ with $w(0)= 0$
such that $f=g\circ w$.

\begin{theorem}\label{th2.4}
 Let $\lambda_1$ and $\lambda_2$ be given by
 \begin{align*}
\lambda_1&=\frac{2n+4(2p-1)}{4+n-2p+\sqrt{16n+n^2 +32p-12np-28p^2}},  \\
\lambda_2&= \frac{2n+4(2p-1)}{-n+2p+\sqrt{16-8n+n^2 -48p+4np+36p^2}},\end{align*}
and   $\lambda_1<\lambda<\lambda_2$.
If the function $f\in\mathcal{A}_{p,n}$ satisfies the inequality
\begin{equation}\label{hyp4} \RE
\left(1+\frac{zf''(z)}{f'(z)}\right)<\left\{
                                       \begin{array}{ll}
                                         \frac{2(1-p)\lambda^2 +(4+n)\lambda+(2-2p-n)}{2(\lambda+1)}, & \hbox{$\lambda_1<\lambda \leq\frac{p+n}{p}$;}
\\[10pt]
                                         \frac{2(1-p)\lambda^2 +n\lambda+(-2+2p+n)}{2(\lambda-1)}, & \hbox{$\frac{p+n}{p}<\lambda <\lambda_2$,}
                                       \end{array}
                                     \right.
 \end{equation}then
\begin{equation}\label{hyp4-s}\frac{1}{p}\frac{zf'(z)}{f(z)}\prec
 \frac{\lambda(1-z)}{\lambda-z}.\end{equation}
\end{theorem}

\begin{proof}
Let us define $w$ by \begin{equation}\label{w4}
\frac{1}{p}\frac{zf'(z)}{f(z)}=\frac{\lambda
(1-w(z))}{\lambda-w(z)}. \end{equation} By doing the logarithmic
differentiation on \eqref{w4}, we get
\[1+\frac{zf''(z)}{f'(z)}=\frac{p \lambda(1-w(z))}{\lambda-z}-\frac{zw'(z)}{1-w(z)}+\frac{zw'(z)}{\lambda-w(z)}.\]
Assume that there exists a point $z_0 \in \mathbb{D}$ such
that $|w(z_0)|=1$ and $|w(z)|<1$ when $|z|<|z_0|$. By applying Lemma \ref{lem1.1}
as in Theorem \ref{th2.1}, it follows that
\begin{align*}  \RE\left(1+\frac{z_0 f''(z_0)}{f'(z_0)}
\right)&= \RE\left(\frac{p\lambda(1-e^{i\theta})}{\lambda -e^{i\theta}}\right)- \RE\left(\frac{m e^{i\theta}}{1-e^{i\theta}}\right)
+ \RE\left(\frac{ me^{i\theta}}{\lambda-e^{i\theta}}\right)\\
&= \frac{p\lambda(\lambda+1)(1-\cos\theta)}{\lambda^2+1-2\lambda\cos\theta}+\frac{m}{2}+\frac{m(\lambda\cos
\theta-1)}{\lambda^2+1-2\lambda\cos\theta}\\
&= \frac{\lambda+1}{2}(2-p)+\frac{(\lambda^2
-1)[(p+m)-p\lambda]}{2(\lambda^2+1-2\lambda\cos\theta)} \\
&\geq
 \frac{\lambda+1}{2}(2-p)+\frac{(\lambda^2
-1)[(p+n)-p\lambda]}{2(\lambda^2+1-2\lambda\cos\theta)},
\end{align*} which yields the inequality
\begin{equation}\label{3.6}
 \RE \left(1+\frac{z_0
f''(z_0)}{f'(z_0)}\right)\geq\left\{
                                       \begin{array}{ll}
                                         \frac{2(1-p)\lambda^2 +(4+n)\lambda+(2-2p-n)}{2(\lambda+1)}, & \hbox{$\lambda_1<\lambda \leq\frac{p+n}{p}$;}
\\[8pt]
                                          \frac{2(1-p)\lambda^2 +n\lambda+(-2+2p+n)}{2(\lambda-1)}, & \hbox{$\frac{p+n}{p}<\lambda <\lambda_2$.}
                                       \end{array}
                                     \right.
\end{equation} Since \eqref{3.6} obviously contradicts hypothesis
\eqref{hyp4}, it follows that $|w(z)|<1$. This proves the subordination \eqref{hyp4-s}.
\end{proof}

\begin{remark} The subordination \eqref{hyp4-s} 
 can be written in equivalent form as
\[\left|\frac{\lambda(zf'(z)/f(z)-1)}{zf'(z)/f(z)-\lambda}\right|<1,\]
or by further computation, as \[\left|\frac{1}{p}\frac{zf'(z)}{f(z)}-\frac{\lambda}{\lambda+1}\right|<\frac{\lambda}{\lambda+1}.\]
The last inequality shows that $f$ is starlike in $\mathbb{D}$.
\end{remark}

\begin{remark}
When $p=1$ and $n=1$, Theorems \ref{th2.1}--\ref{th2.4} reduce to Theorems~\ref{th1.1} and \ref{th1.2}.
\end{remark}


\begin{thebibliography}{99}
\bibitem{duren} P. L. Duren, {\it Univalent Functions}, Grundlehren der Mathematischen Wissenschaften, 259, Springer, New York, 1983.

\bibitem{good} A. W. Goodman, {\it Univalent Functions}, Mariner, Tampa, FL, 1983.

\bibitem{linnmj} L. J. Lin\ and\ H. M. Srivastava, Some starlikeness conditions for analytic functions, Nihonkai Math. J. {\bf 7} (1996), no.~2, 101--112.

\bibitem{liusjm} J. L. Liu, A remark on certain multivalent functions, Soochow J. Math. {\bf 21} (1995), no.~2, 179--181.

\bibitem{living} A. E. Livingston, $p$-valent close-to-convex functions, Trans. Amer. Math. Soc. {\bf 115} (1965), 161--179.

\bibitem{miller} S. S. Miller\ and\ P. T. Mocanu, {\it Differential Subordinations}, Monographs and Textbooks in Pure and Applied Mathematics, 225, Dekker, New York, 2000.

\bibitem{nuno1} M. Nunokawa, S. Owa\ and\ H. Saitoh, Notes on multivalent functions, Indian J. Pure Appl. Math. {\bf 26} (1995), no.~8, 797--805.

\bibitem{nuno2} M. Nunokawa, On the multivalent functions, Indian J. Pure Appl. Math. {\bf 20} (1989), no.~6, 577--582.

\bibitem{nuno3} M. Nunokawa\ and\ S. Owa, On certain subclass of analytic functions, Indian J. Pure Appl. Math. {\bf 19} (1988), no.~1, 51--54.

\bibitem{nuno4} M. Nunokawa, On the theory of multivalent functions, Tsukuba J. Math. {\bf 11} (1987), no.~2, 273--286.

\bibitem{owa} S. Owa, M. Nunokawa, H. Saitoh and H. M. Srivastava, Close-to-convexity, starlikeness, and convexity of certain analytic functions,
Appl. Math. Lett. {\bf 15} (2002), no.~1, 63--69.

\bibitem{owasin} S. Owa, On sufficient conditions for close-to-convex functions of order alpha, Bull. Inst. Math. Acad. Sinica {\bf 20} (1992), no.~1, 43--51.

\bibitem{owabam} S. Owa, On Nunokawa's conjecture for multivalent functions, Bull. Austral. Math. Soc. {\bf 41} (1990), no.~2, 301--305.

\bibitem{owaaml} S. Owa, M. Nunokawa\ and\ H. M. Srivastava, A certain class of multivalent functions, Appl. Math. Lett. {\bf 10} (1997), no.~2, 7--10.

\bibitem{singh} R. Singh\ and\ S. Singh, Some sufficient conditions for univalence and starlikeness, Colloq. Math. {\bf 47} (1982), no.~2, 309--314 (1983).

\bibitem{sriv} H. M. Srivastava and S. Owa, Editors, {\it Current topics in analytic function theory}, World Sci. Publishing, River Edge, NJ, 1992.

\bibitem{umetoh} T. Umezawa, On the theory of univalent functions, T\^{o}hoku Math. J. (2) {\bf 7} (1955), 212--228.

\bibitem{umepam} T. Umezawa, Multivalently close-to-convex functions, Proc. Amer. Math. Soc. {\bf 8} (1957), 869--874.

\bibitem{yama} R. Yamakawa, A proof of Nunokawa's conjecture for starlikeness of $p$-valently analytic functions, in {\it Current topics in analytic function theory}, 387--392, World Sci. Publ., River Edge, NJ.
\end{thebibliography}
\end{document}